\documentstyle[11pt,amssymb,amsmath,amscd,amsbsy,amsfonts,theorem,hyperref,accents,fontenc]{article}

\input xy
\xyoption{all}

\newcommand{\zz}{{\Bbb Z}}

\newcommand{\nn}{\Bbb N}

\newcommand{\pp}{{\Bbb P}}

\newcommand{\ff}{{\Bbb F}}

\newcommand{\ddim}{\operatorname{dim}}

\newcommand{\op}[1]{\operatorname{#1}}

\newcommand{\la}{\langle}
\newcommand{\ra}{\rangle}
\newcommand{\row}{\rightarrow}

\renewcommand{\leq}{\leqslant}
\renewcommand{\geq}{\geqslant}

\newcommand{\nichego}[1]{}

\newcommand{\ov}[1]{\overline{#1}}

\newcommand{\wt}[1]{\widetilde{#1}}

\newcommand{\Ch}{\operatorname{Ch}}

\newcommand{\CH}{\operatorname{CH}}

\newcommand{\dmk}{\op{DM}(k)}

\newcommand{\dmkF}[1]{\op{DM}(k;#1)}
\newcommand{\dmEF}[2]{\op{DM}(#1;#2)}

\newcommand{\dmELF}[3]{\op{DM}({#1}/{#2};{#3})}

\newcommand{\nump}{\,\stackrel{\scriptscriptstyle{Num(p)}}{\sim}\,}

\newcommand{\dmeffkF}[1]{\op{DM}^{-}_{eff}(k,F)}

\newcommand{\nnorm}[1]{|\!|{#1}|\!|}

\newcommand{\Qed}{\hfill$\square$\smallskip}
\newcommand{\Red}{\hfill$\triangle$\smallskip}

\newenvironment{proof}{\noindent{\it Proof}:}{\vskip 5mm}

\newtheorem{proposition}{Proposition}[section]{\bf}{\it}
\newtheorem{theorem}[proposition]{Theorem}{\bf}{\it}
\newtheorem{lemma}[proposition]{Lemma}{\bf}{\it}
{\bf}{\it}
\newtheorem{definition}[proposition]{Definition}{\bf}{\rm}
\newtheorem{conj}[proposition]{Conjecture}{\bf}{\it}
{\bf}{\it}
{\bf}{\it}
\newtheorem{remark}[proposition]{Remark}{\bf}{\rm}
{\bf}{\rm}
{\bf}{\it}
{\bf}{\it}
{\bf}{\it}

{\bf}{\it}
\newtheorem{corollary}[proposition]{Corollary}{\bf}{\it}
{\bf}{\it}

\begin{document}

\title{On isotropic and numerical equivalence of cycles}
\author{Alexander Vishik}
%\address{School of Mathematical Sciences, University of Nottingham, University Park, Nottingham, NG7 2RD, UK}
\date{}
\maketitle

\begin{abstract}
We study the conjecture claiming that, over a flexible field, {\it isotropic Chow groups}
coincide with {\it numerical Chow groups} (with $\ff_p$-coefficients).
This conjecture is essential for understanding the structure of the isotropic motivic category and that of the tensor triangulated spectrum of Voevodsky category of motives. We prove the conjecture for the new range of cases. In particular, we show that, for a given variety $X$,
it holds for sufficiently large primes $p$. We also prove
the $p$-adic analogue. This permits to interpret integral
numerically trivial classes in $\CH(X)$ as $p^{\infty}$-anisotropic ones.
\end{abstract}

\section{Introduction}

The difference in complexity between algebraic geometry and topology can be visualised by the fact that in topology there is only one kind of a point, while in algebraic geometry there are many different types of points, namely, the spectra of all possible finitely generated extensions of the base field. It is a natural idea to supply algebro-geometric objects with their ``local'' versions parametrised by points of various kinds, versions whose complexity would be comparable to that of topological objects. With this in mind, the local versions of the Voevodsky motivic category $\dmkF{\ff_p}$ were introduced in the article \cite{Iso}. These, so-called, {\it isotropic motivic categories} provide realisation functors for ``global'' motives and are parametrized by finitely
generated extensions of $k$. More precisely, the isotropic motivic category $\dmELF{F}{F}{\ff_p}$ is obtained from
the global category $\dmEF{F}{\ff_p}$ by killing the motives of all $p$-{\it anisotropic} varieties (the idea to use such a localisation goes back to T.Bachmann - \cite{BQ}). Following \cite{Iso}, for any field $E$, we denote as $\wt{E}=E(t_1,t_2,\ldots)=E(\pp^{\infty})$ the {\it flexible closure} of it. Given a base field $k$, we may introduce the natural partial ordering (depending on the prime $p$) on the set of finitely generated extensions
$E/k$. Namely $E/k\geq F/k$, if there is a correspondence
$Q\rightsquigarrow P$ of degree prime to $p$, where $E=k(Q)$ and $F=k(P)$. Then we obtain the family of {\it isotropic realization functors}:
$$
\psi_{E,p}:\dmk\row\dmELF{\wt{E}}{\wt{E}}{\ff_p},
$$
where $E/k$ runs over all equivalence classes of finitely generated extensions of $k$ under the equivalence 
defined by the ordering above. The passage to the flexible
closure is needed here to make the target category ``small''. Indeed, in the case of a flexible field, it is expected that the isotropic motivic category should be, 
in various respects, reminiscent of the topological motivic
category $D(\ff_p)$. In particular, homs between compact objects should be finite groups and Balmer's tensor triangulated spectrum of this category should consist of
a single point. 
But aside from similarities, there are visible differences. For example, it was computed for $p=2$ - \cite[Theorem 3.7]{Iso} that the {\it isotropic motivic cohomology} of a point form an external
algebra with generators - duals of Milnor's 
operations. This, in turn, leads to the computation of the
{\it stable isotropic homotopy groups of spheres} by F.Tanania, who identified these groups with the $E_2$-term
of the classical Adams spectral sequence - \cite{ISHG}.

Another direction from which one may try to approach isotropic motives is the case of {\it pure motives}.
The homs between the respective {\it isotropic Chow motives} are described by {\it isotropic Chow groups}
$\Ch^*_{k/k}$.
These groups are obtained from the usual Chow groups
by moding out cycles coming from $p$-anisotropic varieties
- see \cite[2.1]{Iso}.
The central here is \cite[Conjecture 4.7]{Iso} claiming that, in the case of a flexible field, such groups coincide with numerical Chow groups (with finite coefficients):

\begin{conj} 
 \label{Main-conj}
 Let $k$ be a flexible field. Then 
 $\Ch^*_{k/k}=\Ch^*_{Num}$. 
\end{conj}

It implies 
that (in the case of a flexible field) the category 
of isotropic Chow motives is equivalent to the category of
{\it numerical Chow motives} (with $\ff_p$-coefficients).
This, in turn, shows that {\it isotropic realization functors} $\psi_{E,p}$ should provide ``points'' for the triangulated
spectrum $\op{Spc}(\op{DM}_{gm}(k))$ in the sense of Balmer (\cite{Bal}) of (the compact part of) the
Voevodsky motivic category. Thus, we get many new points of the spectrum.

In \cite{Iso} Conjecture \ref{Main-conj} was proven for varieties of dimension $\leq 5$,
for divisors and for cycles of dimension $\leq 2$.
The principal aim of the current article is to expand the range of this result. Namely, 
in Theorem \ref{Main-thm} we will show that it still
holds for varieties of dimension $\leq 2p$ as well as for cycles of dimension $<p$. In particular, for any given variety $X$,
the conjecture is true for sufficiently large $p$. 

In comparison to \cite[Theorem 4.11]{Iso}, our method is
much more transparent. It is based on the use of $K_0$, as well as Adams and Steenrod operations. The Conjecture is equivalent to the fact that any $p$-numerically trivial
class is $p$-anisotropic. For divisors, this is clear, since the Chow groups of the variety surject to the Chow groups of the generic representative of a very ample linear system. So, if the system is numerically trivial, then this
generic divisor is anisotropic. At the same time, for a flexible field, there is no need to distinguish between $k$ and any 
finitely generated purely transcendental extension of it, so we have an ``analogue'' of our generic representative
already over $k$, which will be anisotropic too. This immediately extends to numerically trivial complete intersections. For an arbitrary numerically trivial cycle $u\in\Ch^r$, our strategy is to find $V\in K_0(X)$, such 
that $c_r(V)=u$ and all smaller classes are trivial. This can be done as long as $r\leq p$. Then using Adams operations we annihilate numerically all higher Chern classes of $V$, which is possible, if $\ddim(X)<rp$. 
Our element $u$ now is a part of a numerically trivial
total Chern class of some $V\in K_0(X)$. It remains to prove that such a total Chern class is always anisotropic.
This is done in Proposition \ref{cV-num-is} with the help
of a nice geometric construction of Section 
\ref{main construction} by induction on the dimension of a vector bundle $V$, using the case of a complete intersection. 

The result is easily extended to the $p$-primary and $p$-adic case. The only tool we loose in such passage are Steenrod operations, which affects the range. In Theorems
\ref{thm-primary} and \ref{thm-p-adic} we prove $p$-primary
and $p$-adic analogues of Conjecture \ref{Main-conj} for 
$\ddim(X)\leq p$ and for cycles of dimension $<(p-1)$.
This permits to interpret the integral numerically trivial
classes in $\CH(X)$ over any field $k$ (of characteristic zero) as $p^{\infty}$-anisotropic ones over it's flexible closure, where $p$ is any prime $\geq\ddim(X)$ - see Corollary \ref{p-nezav}. Thus, isotropic Chow groups can be used to approach (the classical) numerical equivalence of cycles with rational coefficients.

 \medskip

\noindent
{\bf Acknowledgements:}
The support of the EPSRC standard grant EP/T012625/1 is gratefully acknowledged. I'm grateful to the Referee for various useful suggestions
which improved the text.

\section{Isotropic Chow groups}

Everywhere below $k$ will be a field of
characteristic zero.

Let $n\in\nn$ and $\Ch^*=\CH^*/n$ be Chow groups modulo $n$. 

\begin{definition}
 Let $n\in\nn$ and $X$ be a scheme of finite type over $k$. 
 We say that $X$ is
 $n$-anisotropic, if degrees of all closed points on $X$
 are divisible by $n$.
\end{definition}

\begin{definition}
 Let $n\in\nn$, $X$ be a scheme over $k$ and $x\in\CH_r(X)$. Then
 $x$ is $n$-anisotropic, if there exists a proper map 
 $f:Y\row X$ from an $n$-anisotropic scheme $Y$ and $y\in\CH_r(Y)$ such that $f_*(y)=x$.
\end{definition}

Now we can introduce {\it isotropic Chow groups}:

\begin{equation}
\label{iso-Ch-gr}
 \Ch_{k/k}:=\Ch/(\,\text{$n$-anisotropic classes}\,).
\end{equation}

This defines an oriented cohomology theory with localisation in the sense of \cite[Definition 2.1]{SU}.
These groups describe homs in the category of {\it isotropic Chow motives} - see \cite[Proposition 5.4]{Iso}. 

From now on until Section \ref{main construction}, we will assume that $n=p$ is prime.
Recall from
\cite[Definition 2.17]{Iso} that the category of {\it isotropic Chow motives} is
the full subcategory of $\dmELF{k}{k}{\ff_p}$ given by the
image of $Chow(k,\ff_p)$ under the isotropic realisation
functor $\dmkF{\ff_p}\row\dmELF{k}{k}{\ff_p}$. 

For an arbitrary field $k$, the isotropic motivic category
may be pretty large. In particular, for algebraically closed field, it coincides with the original ``global''
motivic category (as there are no anisotropic varieties,
in this case). But there is a large class of, so-called, {\it flexible} fields for which isotropic realisation really simplifies things. 
A field $k$ is called {\it flexible}, if it is a purely
transcendental extension of infinite transcendence degree
$k=k_0(t_1,t_2,\ldots)$ of some other field. Any field can
be embedded into its {\it flexible closure} by adjoining free parameters. Such an embedding is conservative on motives.

For a smooth projective variety 
$X\stackrel{\pi}{\row}\op{Spec}(k)$, there is a natural degree pairing:
$$
\Ch^*(X)\times\Ch^*(X)\row\ff_p,
$$
given by $\la a,b\ra=\pi_*(a\cdot b)$.
Since degrees of all zero-cycles on anisotropic varieties
are divisible by $p$, it descends to $\Ch_{k/k}^*$.
Thus, we obtain a surjective map
$\Ch^*_{k/k}\twoheadrightarrow\Ch^*_{Num}$ to Chow groups modulo numerical equivalence (with $\ff_p$-coefficients).
According to Conjecture \ref{Main-conj}, this map should be an isomorphism, if $k$ is flexible.

This Conjecture implies that, for a flexible $k$, the category of isotropic Chow
motives $Chow(k/k;\ff_p)$ should be equivalent to the category
of numerical Chow motives $Chow_{Num}(k,\ff_p)$. The latter
additive category is semisimple. Moreover, since the degree pairing is defined on the level of topological realization,
the group of homs between two objects of the numerical category is a subquotient of homs of their topological realisations. In particular, such groups are finite.
The above Conjecture implies that the same properties should hold for isotropic Chow motives. 

In \cite[Theorem 4.11]{Iso} the Conjecture was proven in the following cases: for
varieties of dimension $\leq 5$, for divisors and for 
cycles of dimension $\leq 2$.
Our aim is to prove an extension of this result:

\begin{theorem}
 \label{Main-thm}
 The Conjecture \ref{Main-conj} holds for:
 \begin{itemize}
  \item[$(1)$] $\ddim(X)\leq 2p$; 
  \item[$(2)$] $\Ch_r$, with $r<p$;
  \item[$(3)$] $\Ch^r(X)$, for $r\leq p$ and $\ddim(X)<p^2-p+r$.
 \end{itemize}
\end{theorem}

In particular, for a given variety $X$, the Conjecture is
valid for sufficiently large $p$.

\section{An outline of the proof}

The Conjecture \ref{Main-conj} is equivalent to the fact
that every numerically trivial class $u$ belongs to the 
subgroup generated by push-forwards from anisotropic varieties. 

We start with the main case: $u\in\Ch^r(X)$, where $r\leq p$ and $\ddim(X)<rp$.
If $Z\subset X$ is a closed subvariety of pure co-dimension $r$, then 
the $r$-th Chern class of $O_Z$ is equal to $(-1)^{r-1}(r-1)![Z]$,
while smaller Chern classes are zero (as $O_Z$ belongs to 
the $r$-th term of the topological filtration on $K_0(X)$).
Thus, if $r\leq p$, then $(r-1)!$ is invertible and there is $V\in K_0(X)$ such that
$c_r(V)=u$ and $c_i(V)=0$, for $i<r$. 

The next step is to make all Chern classes of $V$ numerically trivial. This is achieved by applying an appropriate linear combination of Adams operations to $V$.
Here it is essential that $\ddim(X)<rp$. As a result,
our class $u$ is one of the Chern classes of some $V\in K_0(X)$ whose total Chern class is numerically trivial.
In such a situation, the total Chern class of $V$ will always be anisotropic. 

Since we work modulo $p$, we may assume that $V$ is a 
vector
bundle. Passing to the variety of flags of $V$, we may assume that $V$ is equal to a direct sum of very ample line bundles.
In other words, we may assume that there is a set of very
ample divisor classes $a_i,\,i\in\ov{N}=\{1,\ldots,N\}$ 
such that $V=\oplus_{i=1}^N O(a_i)$. We will choose concrete cycles representing the above divisor classes.
Since the ground field is flexible, we can always replace it by any finitely generated purely transcendental extension. So, we may think of our divisors as of generic
representatives $A_i$ of the respective linear systems.
In particular, $\cup_i A_i$ is a divisor with strict normal crossings. 

Now, we will replace our data $(X; a_i, i\in\ov{N})$ by
$(\wt{X};b_j, j\in\ov{M})$, where 
$\pi:\wt{X}\row X$ is a blow up and 
$\pi^*c_{\bullet}(V)=c_{\bullet}(\wt{V})$, where
$\wt{V}=\oplus_{j=1}^M O(b_j)$. We will call it the {\it Main construction}. Namely, $\pi$ is the blow up in all the
intersections $A_I=\cap_{i\in I}A_i$, $I\subset\ov{N}$. 
If $\wt{\rho}_I$ is the class
of the proper pre-image of $A_I$ under $\pi$, then one
can choose (the very ample classes equal modulo $p$ to)
$b_j:=\sum_{|I|\geq j}\wt{\rho}_I$ and 
check that the elementary symmetric functions of $b_j$s are
$\pi^*$ of those of $a_i$s. This time, though, 
the number of $b_j$s
is no more than $\ddim(X)$. Moreover, the top Chern class of $\tilde{V}$ is a complete intersection which is still 
numerically trivial. By the standard arguments, the generic representative of this complete intersection is anisotropic. Applying the {\it Main construction} again we
see that, modulo anisotropic classes, we may replace $V$
by a vector bundle of dimension one less (as the top divisor class of the new arrangement will be exactly the class of the pre-image of the mentioned complete intersection $B_{\ov{M}}=\cap_j B_j$, and so, an anisotropic class). It remains to apply the induction on $\ddim(V)$. 

This proves the result for $\Ch^r(X)$, where $r\leq p$
and $\ddim(X)<rp$. Multiplying $X$ by $\pp^{p-r}$ and passing to cycles of co-dimension exactly $p$, we obtain the whole case (3) of the Theorem.  
What remains is the case of $\Ch_r(X)$,
where $r<p$ and $\ddim(X)>2r+1$. Here we show using the
moving techniques of \cite[Section 6]{Iso} that, after
an appropriate blow up, our numerically trivial class $u$
is supported on some smooth divisor $D$ and is numerically trivial already on $D$. The result follows by the induction
on $\ddim(X)$ from the main case above. This proves (2). 
Finally, (1) follows from (2) and (3), aside from the case
$p=2$, $\ddim(X)=4$, which is straightforward.

\section{Adams and Steenrod operations}
The purpose of this section is to annihilate numerically
the Chern classes of our $V\in K_0(X)$ while keeping
the $r$-th one: $c_r(V)=u$.  

On $K_0$ we have the action of the Adams operations
$\Psi_m,\,m\in\zz$, originally introduced in \cite{Ad}. Here $\Psi_m$ is a multiplicative operation characterized by the action 
$\Psi_m(L)=L^{\otimes m}$ on a line
bundle $L$. 

Let us describe, how these operations affect characteristic classes of vector bundles.

\begin{proposition}
 \label{psi-act}
 Let $\Psi=\sum_i\Psi_{m_i}$ and $V'=\Psi(V)$. Then
 $$
 c_d(V')=\left(\sum_i m_i^d\right)\cdot c_d(V)+\,\text{polynomial in }\, c_l(V),\,
 l<d.
 $$
\end{proposition}

\begin{proof}
 Since $\Psi_m$ multiplies the Chern roots by $m$,
 we have $c_k(\Psi_m(V))=m^k\cdot c_k(V)$. The rest follows
 from the Cartan's formula.
 \Qed
\end{proof}

\begin{lemma}
 \label{mi-s}
 Suppose $(p-1)\,/\hspace{-1.6mm}|\,(d-r)$. Then there is a (finite)
 collection of integers $\{m_i|i\in I\}$ such that
 $\displaystyle\sum_i m_i^r\equiv 1\,(mod\, p)$ and
 $\displaystyle\sum_i m_i^d\equiv 0\,(mod\, p)$.
\end{lemma}

\begin{proof}
 Since $(d-r)$ is not divisible by $(p-1)$, we can find $m$
 such that $m^d\not\equiv m^r\,(mod\,p)$. Combining $x$ copies of $m$ with $y$ copies of $1$, where $x$ and $y$ are any natural numbers such that $x\, (mod\,p)$ is the inverse of $(m^r-m^d)$ and $y\, (mod\, p)$ is $-x\cdot m^d$, we get the needed collection.
 \Qed
\end{proof}

Let us now investigate the action of Steenrod operations
on characteristic classes.
Such characteristic classes are parametrized by the ordered collections $\vec{r}=(r_1,r_2,\ldots)$ of non-negative integers, almost all zeroes. 
To such a collection we can assign two numbers: the degree
$\nnorm{\vec{r}}=\sum_ii\cdot r_i$ and the inner degree
$|\vec{r}|=\sum_i r_i$. 
The characteristic class $c_{\vec{r}}(V)$ is the coefficient
at ${\vec{d}}^{\vec{r}}=\prod_{i\geq 1}d_i^{r_i}$ of
$\prod_l(\sum_{i\geq 0}d_i\cdot\lambda_l^{i})$,
where $\lambda_l$ are {\it Chern roots} of $V$, where we assume
$d_0=1$ - cf. \cite[(2.3)]{Qu71}. Recall that {\it Chern roots} of $V$ are ``formal'' elements
$\lambda_l$ such that the Chern classes $c_j(V)$ are the elementary
symmetric functions $sym_j(\{\lambda_l\})$ of these ($\lambda_l$s can be realised as divisor classes on the variety $\op{Flag}(V)$ of complete flags of $V$ - see \cite[2.3]{SU}). The mentioned coefficient $c_{\vec{r}}(V)$ is a symmetric function on roots, and so, 
by the Theorem on Symmetric Functions,
is expressible as a polynomial in Chern classes $c_j$s.
It belongs to $\CH^{\nnorm{\vec{r}}}$.
In particular, the usual Chern class $c_r$ corresponds
to the collection $(r,0,0,\ldots)$.

Express $c_{\vec{r}}$ as a polynomial in $c_j$s.
Let us denote as $\kappa(\vec{r})$ the coefficient at $c_n$ (where $n=\nnorm{\vec{r}}$) in this 
polynomial. In other words,
$$
c_{\vec{r}}=\kappa(\vec{r})\cdot c_n+\,\text{decomposable
terms}.
$$
Let 
$\displaystyle\binom{|\vec{r}|}{r_1,r_2,...,r_k}=
\frac{|\vec{r}|!}{r_1!r_2!\cdot\ldots\cdot r_k!}$
be the multinomial coefficient.

\begin{proposition}
 \label{koeff-kap}
 $$
 \kappa(\vec{r})=(-1)^{\nnorm{\vec{r}}-|\vec{r}|}
 \binom{|\vec{r}|}{r_1,r_2,...,r_k}\cdot
 \frac{\nnorm{\vec{r}}}{|\vec{r}|}.
 $$
\end{proposition}

\begin{proof}
 Let $\mu_m$ be the ``roots'' of $d_i$s. That is,
 $d_i=sym_i(\{\mu_m\})$ is the $i$-th elementary symmetric function. Then
 $$
 \prod_l(\sum_id_i\lambda_l^i)=\prod_{l,m}(1+\lambda_l\mu_m)=\prod_m(\sum_jc_j\mu_m^j)=:
 \sum_{\vec{r},\vec{s}}\vec{d}^{\vec{r}}\vec{c}^{\vec{s}}
 \cdot\kappa(\vec{r},\vec{s}).
 $$
 From the $\vec{r}-\vec{s}$ symmetry, the coefficient
 $\kappa(\vec{r})$ at the monomial
 $\vec{d}^{\vec{r}}\cdot c_{\nnorm{\vec{r}}}$ is the same
 as that at the monomial $d_{\nnorm{\vec{r}}}\cdot\vec{c}^{\vec{r}}$. In other words, it is the coefficient at
 $\vec{c}^{\vec{r}}$ in the additive characteristic class
 $\sum_l\lambda_l^n$, where $n=\nnorm{\vec{r}}$.
 It is given by the Girard's formula \cite[Problem 16-A]{MS}. We have:
 \begin{equation*}
  \begin{split}
  \sum_{j\geq 1}\left(\sum_l\lambda_l^j\right)\cdot t^j=
  \sum_l\frac{\lambda_lt}{1-\lambda_lt}=
  \frac{-t\left(\prod_l(1-\lambda_lt)\right)'}{\prod_l(1-\lambda_lt)}=
  \frac{-\sum_j(-1)^jjc_jt^j}{\sum_j(-1)^jc_jt^j}.
  \end{split}
 \end{equation*}
This fraction is equal to
$$
\sum_{\vec{r}\neq\vec{0}}(-1)^{\nnorm{\vec{r}}-|\vec{r}|}
\binom{|\vec{r}|}{r_1,r_2,...,r_k}\cdot\frac{\nnorm{\vec{r}}}{|\vec{r}|}\cdot\vec{c}^{\vec{r}}\cdot t^{\nnorm{\vec{r}}},
$$
as can be seen by observing that if we denote as $c_{\bullet}(t)=\sum_{r\geq 1}c_rt^r$ the {\it $t$-normalized total Chern class}, then 
$$
\op{log}(1+c_{\bullet}(-t))=-\sum_{n\geq 1}\frac{(-1)^nc_{\bullet}(-t)^n}{n}=
-\sum_{n\geq 1}\sum_{\vec{r};|\vec{r}|=n}(-1)^{\nnorm{\vec{r}}-|\vec{r}|}\binom{|\vec{r}|}{r_1,\ldots,r_k}\cdot\frac{\vec{c}^{\vec{r}}}{|\vec{r}|}t^{\nnorm{\vec{r}}},
$$
after which it remains only to differentiate and multiply by $-t$.
 \Qed
\end{proof}

It will be convenient to use an alternative slightly different notation
for collections. Namely, we will denote the collection
$\vec{r}=(r_1,r_2,\ldots)$ also as $(1^{r_1},2^{r_2},\ldots)$, 
where we will list only terms with non-zero $r_i$s.
To avoid confusion, we will denote the respective Chern
classes as $s_{1^{r_1},2^{r_2},...}=c_{\vec{r}}$.

We have the action of Steenrod operations
$P^i:\Ch^*\row\Ch^{*+i(p-1)}$ on (mod $p$) Chow groups - \cite{Br,VoOP}.

\begin{proposition}
 \label{St-Ch}
 The Chern class $c_d(V)$ belongs to the subgroup generated
 by products of smaller Chern classes and images of Steenrod
 operations (on smaller Chern classes), for all
 $d\neq l\cdot p^t$, for $1\leq l<p$.
\end{proposition}

\begin{proof}
 Recall that $c_d=s_{1^d}$. Since the {\it Total Steenrod
 operation} $P^{Tot}=\sum_{k\geq 0}P^k$ acts on Chern roots by 
 $P^{Tot}(\lambda)=(\lambda+\lambda^p)$ (\cite[Remark 8.5]{Br}), we see that
 $P^k(s_{1^{d-k(p-1)}})=s_{1^{d-kp},p^k}$.
 The latter collection $\vec{r}$ has $\nnorm{\vec{r}}=d$
 and $|\vec{r}|=d-k(p-1)$ with only two non-zero terms $r_1=d-kp$ and $r_p=k$. Thus, by Proposition 
\ref{koeff-kap}, the respective Chern class is equal to
\begin{equation*}
 \begin{split}
 &\frac{d}{d-k(p-1)}\binom{d-k(p-1)}{k}\cdot s_{1^d}+
 \,\text{decomp. terms}=
 \frac{d}{k}\binom{d-kp+(k-1)}{k-1}\cdot s_{1^d}+\,\text{decomp. terms}.
 \end{split}
\end{equation*}
Thus, if there 
exists, at least, one $k$ such that the coefficient
$\displaystyle\frac{d}{k}\binom{d-kp+(k-1)}{k-1}$ is not divisible
by $p$, then $s_{1^d}$ belongs to the subgroup generated
by products of smaller Chern classes and by the images of Steenrod operations applied to smaller Chern classes.

Suppose $d$ is such that, for any natural $k$, 
$p\,|\,\displaystyle\frac{d}{k}\binom{d-kp+(k-1)}{k-1}$. 
Looking at $k=1$ we see that either $d<p$, or 
$p\,|\,d$.
In the latter case, let $d=pd_1$. Then the above divisibility holds for all $k$ if and only if it holds for $p\,|\,k$. For $k=pk_1$, the respective coefficient
is $\displaystyle\frac{d_1}{k_1}\binom{(d_1-k_1p)p+(k_1-1)p+(p-1)}{(k_1-1)p+(p-1)}$, $p$-adic valuation of which is equal to that of $\displaystyle\frac{d_1}{k_1}\binom{d_1-pk_1+(k_1-1)}{k_1-1}$. Thus, $d_1$ has the same property as
$d$. By induction, we obtain that then $d=l\cdot p^t$,
for some $l<p$ and $t\geq 0$.
 \Qed
\end{proof}

Suppose, $r\leq p$ and $V\in K_0(X)$ is such that $c_r(V)=u$ is our numerically trivial class, while $c_i(V)=0$, for $i<r$. 

\begin{proposition}
 \label{annih-Chern}
 We can replace $V$ (by another element of $K_0(X)$) to make $c_d(V)$ numerically trivial,
 for all $d<rp$, while keeping $c_r(V)=u$.
\end{proposition}

\begin{proof}
Let us prove by induction on $d<rp$ that we can make 
$c_i(V)$ numerically trivial, for $i\leq d$.
By assumption, we have it for $d=r$. Suppose, 
$c_i(V)\nump 0$, for $i<d$. By \cite[Proposition 4.6]{Iso},
Steenrod operations respect numerical equivalence. Hence,
by Proposition \ref{St-Ch}, if $d\neq l\cdot p^t$, for $1\leq l<p$, then $c_d(V)$ is numerically trivial as well.
If $d=l\cdot p^t$, with $1\leq l<p$, but $l\neq r$ (or $l\neq 1$, if $r=p$), then $(d-r)$ is not divisible
by $(p-1)$. Hence, by Proposition \ref{psi-act}
and Lemma \ref{mi-s}, we can find an appropriate linear
combination $\Psi$ of Adams operations which makes
$c_d$ numerically trivial, keeps $c_r$ still equal to $u$
and keeps all $c_i$ numerically trivial, for $i<d$, and zero, for $i<r$.
Hence, we can make the induction step as long as
$\displaystyle\frac{d}{r}$ is not a power of $p$. This is always
so for $d<rp$.
 \Qed
\end{proof}

\section{The Main construction}
\label{main construction}

Let $X$ be a smooth projective variety, $a_i,\,i\in\ov{N}$ 
be very ample divisor classes on it and
$A_i$ be smooth cycles representing them, such that
$\cup_{i\in\ov{N}}A_i$ is a divisor with strict normal crossings.
Denote as $A_I$ the faces $\cap_{i\in I}A_i$ of our divisor, where $I\subset\ov{N}$.

Let $\pi:\wt{X}\row X$ be the blow-up of all $A_I$s,
for $I\subset\ov{N}$. Let $\wt{\rho}_I$ be the class of the 
proper pre-image of $A_I$ and $\rho_I$ be (the pull-back of)
$c_1(O(-1))$ of the $I$-th blow up. Then 
$$
\rho_I=\sum_{I\subset J}\wt{\rho}_J
\hspace{12mm}\text{and}\hspace{12mm}
\pi^*(a_i)=\sum_{i\in I}\wt{\rho}_I.
$$
Let $2^{\ov{N}}$ be the set of subsets of $\ov{N}$.
For any $I\subset\ov{N}$, consider $U_I\subset 2^{\ov{N}}$,
given by $U_I=\{J|I\subset J\}$. Let's close the set of 
these subsets of $2^{\ov{N}}$ under $\cap$ and $\cup$, 
and call the subsets obtained this way {\it good}.
Clearly, these are exactly those subsets which have the 
property that if $I$ is an element, then so is $J$, for all $J\supset I$.
For any {\it good} $U\subset 2^{\ov{N}}$, consider
$\rho_U:=\sum_{J\in U}\wt{\rho}_J$. 
In particular, $\rho_{U_I}=\rho_I$.

\begin{lemma}
 \label{per-ob}
 For any good $U,W\subset 2^{\ov{N}}$ we have:
 $$
 (t-\rho_U)(t-\rho_W)=(t-\rho_{U\cap W})(t-\rho_{U\cup W}).
 $$
\end{lemma}

\begin{proof}
A {\it good} subset containing $I$ will necessarily contain any $J$ with $I\subset J$.
 Let $\wt{E}_I$ be the proper pre-image of $A_I$ under 
 $\pi$. Then $\wt{E}_K$ and $\wt{E}_L$ don't intersect
 unless $K\subset L$, or $L\subset K$.
 Since for any $\wt{\rho}_K$ from $(\rho_U-\rho_{U\cap W})$
 and any $\wt{\rho}_L$ from $(\rho_W-\rho_{U\cap W})$,
 neither $K\subset L$, nor $L\subset K$, we obtain that
 $(\rho_U-\rho_{U\cap W})(\rho_W-\rho_{U\cap W})=0$. 
 Since, by the very definition,
 $\rho_{U\cup W}=\rho_U+\rho_W-\rho_{U\cap W}$ (recall, that
 $\rho_U=\sum_{J\in U}\wt{\rho}_J$), we get
 the formula.
 \Qed
\end{proof}

\begin{remark}
 \label{dvum-knol}
 In other words, the 2-dimensional vector bundles
 $O(\rho_U)\oplus O(\rho_W)$ and
 $O(\rho_{U\cap W})\oplus O(\rho_{U\cup W})$ have the
 same Chern classes in $\CH^*$. Moreover, if $A^*$ is an oriented cohomology theory and we denote
 as $\rho_U^A=c_1^A(O(\rho_U))$ and $\wt{\rho}_J^A=c_1^A(O(\wt{\rho}_J))$ the first 
 $A$-Chern classes of the respective line bundles, then 
 $\rho_U^A$ is the $A$-formal sum $\sum_{J\in U}^A\wt{\rho}_J^A$ and
 the 
 identities $\rho^A_{U\cup W}=\rho^A_U+\rho^A_W-\rho^A_{U\cap W}$ and
 $(\rho^A_U-\rho^A_{U\cap W})(\rho^A_W-\rho^A_{U\cap W})=0$ still hold. 
 To see it, it is sufficient to observe that $\rho_{U\backslash U\cap W}^A$ and $\rho_{W\backslash U\cap W}^A$ have disjoint support.
 Thus, our 2-dimensional vector bundles have the same Chern classes in any oriented cohomology theory. In particular, looking at $c_1^{K_0}$, we get that our vector bundles coincide as elements of $K_0$. 
 \Red
\end{remark}

Let $U_k\subset 2^{\ov{N}}$, $k\in\ov{m}$ be a collection
of good subsets. For any $P\subset\ov{m}$ denote as
$U_P$ the intersection $\cap_{k\in P}U_k$, and for 
$l\in\ov{m}$, denote as
$W_l$ the subset $\cup_{P\subset\ov{m};|P|=l}U_P$.
Let $\displaystyle V'=\oplus_{k\in\ov{m}}O(\rho_{U_k})$
and $\displaystyle V''=\oplus_{l\in\ov{m}}O(\rho_{W_l})$.
Then we have:

\begin{proposition}
 \label{lem-dva}
 In the above situation, $[V']=[V'']$ in $K_0(\wt{X})$ and
 $$
 \prod_{k\in\ov{m}}(t-\rho_{U_k})=
 \prod_{l\in\ov{m}}(t-\rho_{W_l}).
 $$
\end{proposition}

\begin{proof}
Induction on $m$.
For $m=1$, it is a tautology.
The induction step follows from Lemma \ref{per-ob}
and Remark \ref{dvum-knol}.
 \Qed
\end{proof}

Let $\Ch^*=\CH^*/n$, for some $n\in\nn$.
Consider $U_i=U_{\{i\}}$, $i\in\ov{N}$. Then
$\rho_{U_i}=\pi^*(a_i)$, while
$\rho_{W_j}=\sum_{J\subset\ov{N};|J|\geq j}\wt{\rho}_J$.
Let us choose some very ample representative $b_j$ for
$\rho_{W_j}\in\Ch^1(\wt{X})$. Then it follows from
Proposition \ref{lem-dva} that 
for $V=(\oplus_{i\in\ov{N}}O(a_i))$ and
$\wt{V}=(\oplus_{j\in\ov{N}}O(b_j))$, the bundles 
$\pi^*(V)$ and $\wt{V}$ have the same Chern classes in $\Ch^*$.

Now we replace the data $(X;A_i, i\in\ov{N})$ by
$(\wt{X}; B_j, j\in\ov{N})$, where we choose some smooth
representatives of our classes transversal to each other. 
Note that, in reality,
some of these classes $b_j$s may be zero (modulo $n$) or anisotropic. In particular, by the very construction, $b_j=0$, for $j>\ddim(X)$,
and so, we can replace $\wt{V}$ by a vector bundle of 
dimension $\leq\ddim(X)$. 
Moreover, the top Chern class of this vector bundle is 
represented by a complete intersection (of very ample divisors).

\section{Proof of the Main Theorem}

Now we are ready to prove Theorem \ref{Main-thm}.
We will first prove part $(3)$, which is the main case.

Suppose $r\leq p$, $\ddim(X)<rp$ and $u\in\Ch^r(X)$
is a numerically trivial class. 
For any subvariety $Z\subset X$ of co-dimension $r$,
$O_Z$ belongs to the $r$-th component of the topological filtration on $K_0$, it follows that
$c_i(O_Z)=0$, for $i<r$. At the same time, 
$c_r(O_Z)=(-1)^{r-1}(r-1)![Z]$ - \cite[Example 15.3.1]{Fu}. If $r\leq p$, the latter
coefficient is invertible. Thus, we can find $V\in K_0(X)$, such that $c_i(V)=0$, for
$i<r$, and $c_r(V)=u$.

It follows from Proposition \ref{annih-Chern}
that we can additionally make $c_d(V)$ numerically
trivial, for all $d<rp$, that is, for all $d$ (since $\ddim(X)<rp$). Thus, the total Chern class
$c_{\bullet}(V)=c_1+c_2+\ldots$ is numerically trivial.

\begin{proposition}
 \label{cV-num-is}
 Let $k$ be flexible, $n\in\nn$ and $V\in K_0(X)$ be such that
 the total Chern class $c_{\bullet}(V)$ is $n$-numerically trivial. Then $c_{\bullet}(V)$ is $n$-anisotropic.
\end{proposition}

\begin{proof}
Since $1+c_{\bullet}(n^mW)=(1+c_{\bullet}(W))^{n^m}$ and we work with cycles modulo $n$, we may assume that $V$ is a vector bundle. Indeed, the
Chern classes of an arbitrary element $V_1-V_2$ of $K_0$ will be the same
as that of $V_1+(n^m-1)V_2$, for sufficiently large $m$ (note that $c_i$s are zero, for $i>\ddim(X)$ and so, $(1+c_{\bullet}(V_2))$ raised to the
power $n$ sufficiently many times will be equal to $1$ modulo $n$).
Passing to the flag variety of $V$ (and using the fact that both $\Ch_{k/k}$ and $\Ch_{Num}$ are
oriented cohomology theories and, in particular, satisfy the projective bundle axiom), we may assume that
$V=\oplus_{i=1}^N O(a_i)$, where $a_i$ are very ample
divisor classes (recall, again, that we work modulo $n$ and so, any class can be made very ample by adding an appropriate $n$-multiple of another class). 

Note that it is sufficient 
to prove that $c_{\bullet}(V)$ becomes anisotropic after some
blow up. Indeed, if $\pi:\wt{X}\row X$ is such a blow up then $\pi^*(c_{\bullet}(V))$ is still numerically trivial, and if it is
anisotropic, then so is $c_{\bullet}(V)=\pi_*\pi^*(c_{\bullet}(V))$.

We choose some smooth mutually transversal 
representatives $A_i$ of our classes $a_i$ and
apply the {\it Main construction}, i.e. replace
the data $(X;A_i, i\in\ov{N})$ by $(\wt{X};B_j, j\in\ov{M})$,
where $\pi:\wt{X}\row X$ is a blow up and
$c_{\bullet}(\wt{V})=\pi^*(c_{\bullet}(V))$. 
In particular, $c_r(\wt{V})=\pi^*(u)$ and all
Chern classes of $\wt{V}$ are still numerically trivial.
This time, $M\leq\ddim(X)$ and the top Chern class
$c_M(\wt{V})=\prod_jb_j$ is a complete intersection.
Now we proceed by induction on $M$ (= dimension of our
vector bundle, which we rename $V$). Indeed, $c_M(V)$ is a numerically
trivial complete intersection. It follows from
\cite[Proposition 4.15]{Iso} (using the flexibility of $k$) 
that by choosing 
appropriate representatives $B_j$ of our divisorial
classes, we can make $\cap_j B_j$ anisotropic.
If we now apply the {\it Main construction} again, then the
$M$-th divisor class in it will be given by the
class of the proper pre-image of $\cap_j B_j$, and so,
will be anisotropic too. So, we can safely remove this
last line bundle, without changing the image of $u=c_r(V)$
in the isotropic Chow group $\Ch^r_{k/k}$. 
Thus, we managed to decrease the dimension of $V$.
The induction step is proven.
\Qed
\end{proof}

The above Proposition proves the case $\Ch^r$, $r\leq p$, 
$\ddim(X)<rp$. 
Let now $r\leq p$ and $\ddim(X)<p^2-p+r$. Consider $Y=X\times\pp^{p-r}$ and $v=u\cdot h^{p-r}\in\Ch^p(Y)$,
where $h=c_1(O(1))$. Then $v$ is still numerically trivial and, since $u=\pi_*(v)$ for the natural projection
$Y\stackrel{\pi}{\row}X$, it is sufficient to show that $v$ is anisotropic. But, this time, the co-dimension is $p$ and
$\ddim(Y)<p\cdot p$, so it is covered by the 
case proven above.
The proof of part $(3)$ is complete. 

Next we will prove part $(2)$. More precisely, we will
show that the Conjecture holds for $\Ch_r(X)$, where
$r<p$ and $\ddim(X)>2r+1$. We will do it by induction
on the $\ddim(X)$.
Since $r$ is less than half the dimension of $X$, by
\cite[Corollary 6.2]{Iso}, after an appropriate blow up,
our numerically trivial class $u$ may be represented by
the class $[S']$ of a smooth subvariety. Blowing up $S'$,
we may assume that $u$ is supported on a smooth (but, possibly, disconnected) divisor $E$. Though, 
the cycle $[S]$ on it
representing (the pull-back of) $u$ may be non-smooth.
We can employ the moving techniques from \cite{Iso}.
As a particular case of \cite[Proposition 6.5]{Iso}
we have:

\begin{proposition}
 \label{moving}
 Let $n\in\nn$ and $X$ be a smooth projective connected variety, $E$ be a divisor on it smooth outside closed $n$-anisotropic subscheme. Let $S\subset E$ be a closed subscheme on it
 of positive co-dimension. Then, over some finitely generated purely transcendental extension of $k$,
 there is an irreducible divisor $Z$ on $X$ such that $[Z]=[E]\in\Ch^1(X)$, $S\subset Z$, $Z$ is smooth outside
 an $n$-anisotropic subscheme on $S$, and the restriction
 $\CH^*(X)\twoheadrightarrow\CH^*(Z\backslash S)$
 is surjective.
\end{proposition}

Since our field is flexible, the same holds already 
over $k$. Thus, we may assume that $u$ is represented
by the class $[S]$ of a closed subvariety, which is
contained in an irreducible divisor $Z$, which is
smooth outside an anisotropic subscheme and such that
the restriction $\Ch^r(X)\twoheadrightarrow
\Ch^r(Z\backslash S)$ is surjective.
Since $\ddim(Z)>2r$ and $\ddim(S)=r$, it follows that 
$\Ch^r(Z)=\Ch^r(Z\backslash S)$. Thus, we obtain the
surjection $\Ch^r(X)\twoheadrightarrow\Ch^r(Z)$.

Using the results of H.Hironaka, we can resolve the 
singularities of $Z$ (inside $X$) by blowing up anisotropic
centers: $\pi:\wt{X}\row X$. Here the proper pre-image
$\wt{Z}$ of $Z$ is smooth connected and the isotropic class
$\pi^*(u)\in\Ch_{k/k;r}(\wt{X})$ is supported on $\wt{Z}$
and represented by the class $[\wt{S}]$ 
of the proper pre-image of $S$
(since the special divisor of $\pi$ is anisotropic).
Moreover, the maps $\pi^*:\Ch^*_{k/k}(X)\stackrel{=}{\row}
\Ch^*_{k/k}(\wt{X})$ and 
$\pi_Z^*:\Ch^*_{k/k}(Z)\stackrel{=}{\row}
\Ch^*_{k/k}(\wt{Z})$
are isomorphisms. Thus, the restriction
$\Ch^r_{k/k}(\wt{X})\twoheadrightarrow\Ch^r_{k/k}(\wt{Z})$
is surjective. Since $\pi^*(u)=[\wt{S}]$ is numerically trivial
on $\wt{X}$, this cycle is numerically trivial already
on $\wt{Z}$ (by the projection formula and surjectivity above). Thus, we managed to decrease the dimension
of our variety (while keeping $r$ fixed). 
There are two cases: 1) $\ddim(Z)>2r+1$; 
2) $\ddim(Z)=2r+1$.
In the first case, we conclude by the inductive assumption.
In the second, since $r<p$, we see that
$\ddim(Z)<2p$ and $u\in\Ch^s(Z)$, where $s\leq p$, so the
result follows from part $(3)$ proven above.

Finally, part $(1)$ follows from parts $(2)$ and $(3)$,
aside from the case $p=2$, $\ddim(X)=4$, where only the  case $r=2$ is non-trivial.
This is a particular case of \cite[Theorem 4.11]{Iso}. Alternatively, by \cite[Theorem 6.1]{Iso}, after an appropriate blow-up, $u$ is a polynomial in divisorial classes, that is, a sum of classes of complete intersections. Then, as above, we can find a direct sum of very ample line bundles $V=\oplus_{i=1}^N O(a_i)$ with $c_2(V)=u$ and $c_1(V)=0$ (sufficient to see it for a complete intersection $[Z]=a\cdot b$, where $[O_Z]=([O]-[O(-a)])([O]-[O(-b)])$). Using Proposition 
\ref{annih-Chern}, we may assume that $c_3(V)$ is numerically trivial as well.
Applying the Main construction,
we reduce to the case where $\ddim(V)\leq 4$, keeping numerical triviality of $c_i(V)$, $i\leq 3$. But since
$c_1(V)=0$, we may assume that $V=U\oplus\Lambda^3 U$,
for some 3-dimensional $U$, and $c_4(V)=c_3(V)\cdot d+
c_2(V)\cdot d^2+d^4$, where $d=c_1(\Lambda^3 U)$. Adding
4 copies of $O(d)=\Lambda^3U$ to $V$, we obtain a vector
bundle with $c_2$ equal $u$ and $c_{\bullet}$ numerically
trivial (note that $\ddim(X)=4<8$). It remains to apply Proposition \ref{cV-num-is}.
Theorem \ref{Main-thm} is proven.
\Qed

\section{Isotropic Chow groups with $p$-primary and $p$-adic coefficients}
\label{primary}

In this Section, $\Ch$ will denote $\CH/p^m$ - Chow groups
with $\zz/p^m$-coefficients. 

In complete analogy with the case of $\ff_p$-coefficients
one may define the {\it isotropic Chow motivic category
with $\zz/p^m$-coefficients} - see \cite[Section 5]{Iso}. The $\op{Hom}$s in this 
category are given by {\it isotropic Chow groups with
$\zz/p^m$-coefficients} $\Ch_{k/k}=\CH_{k/k}(-,\zz/p^m)$ - see (\ref{iso-Ch-gr}). 

For a smooth projective variety 
$X\stackrel{\pi}{\row}\op{Spec}(k)$, we still have a natural degree pairing:
$$
\Ch^*(X)\times\Ch^*(X)\row\zz/p^m,
$$
giving the surjection $\Ch_{k/k}\twoheadrightarrow\Ch_{Num(p^m)}$ to numerical Chow groups
with $\zz/p^m$-coefficients.

\begin{conj}
 \label{conj-primary}
 If $k$ is flexible, then for all $m$,
 $\CH_{k/k}(X,\zz/p^m)$ coincides with $\CH_{Num}(X,\zz/p^m)$. 
\end{conj}

Another advantage of our Theorem \ref{Main-thm} 
in comparison to \cite[Theorem 4.11]{Iso} is that it is
easily extendable to the case of $p$-primary coefficients,
though with some loss of range.

\begin{theorem}
 \label{thm-primary}
 The Conjecture \ref{conj-primary} holds for:
 \begin{itemize}
  \item[$(1)$] $\ddim(X)\leq p$; 
  \item[$(2)$] $\Ch_r$, for $r<p-1$.
 \end{itemize}
 In particular, for any given $X$, the Conjecture is true for sufficiently large $p$.
\end{theorem}

\begin{proof}
 The main point is that all the tools used in the proof of
 Theorem \ref{Main-thm}, aside from Steenrod operations, will work also in the $p$-primary situation.
 
 The Main case is: $u\in\Ch^r(X)$, where $r\leq p$ and 
 $\ddim(X)<r+p-1$. 
 As $(r-1)!$ is invertible in $\zz/p^m$, as above, we can find $V\in K_0(X)$, such that $c_r(V)=u$ and $c_i(V)=0$, for $i<r$. 
 For $r<d<r+p-1$, the difference $(d-r)$ is not divisible
 by $(p-1)$. So, we can annihilate numerically $c_d(V)$ using Proposition \ref{psi-act} and increasing induction on $d$ by applying $p^{m-1}$
 times the linear combination of Adams operations, 
 suggested by Lemma \ref{mi-s}. Since $\ddim(X)<r+p-1$,
 this will make $c_{\bullet}(V)$ numerically-trivial.
 It remains to apply Proposition \ref{cV-num-is}.
 
 To prove (2), we need to show that the result holds for
 $\Ch_r$, where $r<p-1$ and $\ddim(X)>2r+1$. This is done with the help of Proposition \ref{moving} exactly as in
 the proof of Theorem \ref{Main-thm} by induction on the
 $\ddim(X)$. Namely, after an appropriate blow-up, $u$ will be numerically trivial already on some smooth connected divisor $Z$, where either $\ddim(Z)>2r+1$, in which case
 we are done by inductive assumption, or $\ddim(Z)=2r+1$,
 which implies that $\ddim(X)<\op{codim(u)}+(p-1)$ and we are done by the Main case above.
 
 Finally, part (1) follows from (2) and the case of divisors - \cite[Statement 6.3, Proposition 4.15]{Iso}.
 \Qed
\end{proof}

\begin{remark}
 Over flexible fields, 
 $\CH_{k/k}(X,\zz/p^m)$ can be defined as
 $\CH(X)/(p^m-\text{anisotropic classes})$. This is so,  because, over such a field, classes divisible by $p^m$ are   automatically
 $p^m$-anisotropic, since it is so for divisors by
 \cite[Statement 6.3]{Iso}, while up to blow up, any class
 is a polynomial in divisorial ones.
 \Red
\end{remark}

Denote as $\CH(X;\zz_p)$ the limit
$\displaystyle\lim_m\CH(X;\zz/p^m)$, that is, the completion of $\CH(X)$
in the $p$-adic topology. 
For a smooth projective $X$, we have a natural degree pairing:
$$
\CH(X;\zz_p)\times\CH(X;\zz_p)\row\zz_p,
$$
and the numerical version $\CH_{Num}(X;\zz_p)$ obtained by
moding out the kernel of this pairing. This can be extended
to an oriented cohomology theory using \cite[Definition 4.3]{Iso}. Since degree pairing is defined on the level of the topological realization in $H_{Top}(X;\zz_p)=H_{Top}(X;\zz)\otimes\zz_p$
which is a finitely generated $\zz_p$-module,
it follows that $\CH_{Num}(X;\zz_p)$ is a subquotient of
$H_{Top}(X;\zz_p)$, and so, a finitely generated $\zz_p$-module too. As the target of the pairing has no $\zz_p$-torsion, we get that $\CH_{Num}(X;\zz_p)$ is a free
$\zz_p$-module of finite rank. By a similar reason, 
$\CH_{Num}(X)$ is a free $\zz$-module of finite rank and 
the natural map 
$\CH_{Num}(X)\otimes\zz_p\row\CH_{Num}(X;\zz_p)$ is an
embedding, which must be an isomorphism, since it is surjective modulo $p^m$, for any $m$. Thus, 
$\CH_{Num}(X;\zz_p)=\CH_{Num}(X)\otimes\zz_p$ -
the numerical version $A^*_{(\infty)}$ for the free theory $A^*=\CH(X)\otimes\zz_p$ - see 
\cite[Section 4.1]{Iso} (observe, that the theory
$\CH(X;\zz_p)$ itself is not even {\it constant}).
It also shows that $\displaystyle\CH_{Num}(X;\zz_p)=
\lim_m\CH_{Num}(X;\zz/p^m)$.

Similarly, we may define the {\it isotropic Chow groups with $p$-adic coefficients} as the
limit 
$$
\CH_{k/k}(X;\zz_p)=\lim_m\CH_{k/k}(X;\zz/p^m).
$$
For flexible fields, it is the same as the limit
$\displaystyle\lim_m\CH(X)/(p^m-\text{anisotropic classes})$.
Thus, in this case, the image of $\CH(X)$ in $\CH_{k/k}(X;\zz_p)$ will be equal to $\CH(X)/(p^{\infty}-\text{anisotropic classes})$, where a class is $p^{\infty}$-{\it anisotropic}, if it is $p^m$-anisotropic, for any $m$.
There is a natural map
$$
\CH_{k/k}(X;\zz_p)\row\CH_{Num}(X;\zz_p).
$$
The following Conjecture is a consequence of Conjecture \ref{conj-primary}.

\begin{conj}
 \label{conj-p-adic}
 Let $k$ be flexible. Then $\CH_{k/k}(X;\zz_p)=\CH_{Num}(X;\zz_p)$.
\end{conj}

In particular, from Theorem \ref{thm-primary} we obtain:

\begin{theorem}
 \label{thm-p-adic}
 In the situation of Theorem \ref{thm-primary}, the Conjecture \ref{conj-p-adic} holds true. 
 In particular, for $k$ - flexible and $p\geq\ddim(X)$,
 $$
 \displaystyle\CH_{Num}(X)=
 \CH(X)/(p^{\infty}-\text{anisotropic classes}).
 $$
\end{theorem}

\begin{proof}
 Since both $\CH_{k/k}(X;\zz_p)$ and $\CH_{Num}(X;\zz_p)$
 are limits of $\CH_{k/k}(X;\zz/p^m)$ and $\CH_{Num}(X;\zz/p^m)$, respectively, it follows from Theorem \ref{thm-primary} that Conjecture \ref{conj-p-adic} holds true for $\ddim(X)\leq p$, as well as for
 cycles of dimension $<p-1$. 
 In particular, for $p\geq\ddim(X)$ and flexible $k$,
 $$
 \CH_{Num}(X)\otimes\zz_p=\lim_m\CH(X)/(p^m-\text{anisotropic classes}).
 $$
 The image of $\CH(X)$ in the left group is equal to $\CH_{Num}(X)$, while on the right it is equal to
 $\CH(X)/(p^{\infty}-\text{anisotropic classes})$. Hence,
 the equality.
 \Qed
\end{proof}

\begin{remark}
 Since, for a variety $X/k$ and a flexible closure
 $k_{flex}=k(\pp^{\infty})$, we have
 $\CH(X)=\CH(X_{k_{flex}})$, Theorem \ref{thm-p-adic}
 gives a description of the numerical Chow groups over an
 arbitrary field.
 \Red
\end{remark}

\begin{corollary}
 \label{p-nezav}
 Let $k$ be flexible. Then:
 \begin{itemize}
 \item[$(1)$] Numerically trivial classes in $\CH(X)$
 are exactly $p^{\infty}$-anisotropic classes, for
 $p\geq\ddim(X)$.
 \item[$(2)$]
 The subgroup of $p^{\infty}$-anisotropic classes in $\CH(X)$ doesn't depend on the choice of a prime $p\geq\ddim(X)$ and is a subgroup of a finite co-rank.
 \item[$(3)$] Torsion classes in $\CH(X)$ are $p^{\infty}$-anisotropic, for any $p\geq\ddim(X)$.
 \end{itemize}
\end{corollary}

\bigskip

\begin{itemize}
\item[address:] {\small School of Mathematical Sciences, University of Nottingham, University Park, Nottingham, NG7 2RD, UK}
\item[email:] {\small\ttfamily alexander.vishik@nottingham.ac.uk}
\end{itemize}

\end{document}